\documentclass[12pt]{article}

\usepackage{url}
\usepackage{amsthm,amsmath,amssymb}
\usepackage{enumerate}
\usepackage[colorlinks=true,citecolor=black,linkcolor=black,urlcolor=blue]{hyperref}
\usepackage{bm}
\usepackage[colorlinks=true,citecolor=black,linkcolor=black,urlcolor=blue]{hyperref}
\usepackage{graphicx}
\usepackage{subfigure}
\usepackage{tikz}

\newcommand{\Mod}[1]{\ \mathrm{mod}\ #1}
\numberwithin{equation}{section}

\theoremstyle{plain}
\newtheorem{theorem}{Theorem}[section]
\newtheorem{lemma}[theorem]{Lemma}
\newtheorem{corollary}[theorem]{Corollary}
\newtheorem{remark}[theorem]{Remark}

\title{\bf Humps in Motzkin paths and standard Young tableaux in a $(2,1)$-hook}

\author{Xiaomei Chen\\
\small School of Mathematics and Statistics\\
\small Hunan University of Science and Technology\\
\small Xiangtan 411201, China\\
\small\tt xmchen@hnust.edu.cn\\
}

\date{
}
\DeclareMathOperator{\SYT}{SYT}

\begin{document}
\begin{sloppypar}

\maketitle
\begin{abstract}
  We calculate the number of humps and peaks in Motzkin paths with a given height, and calculate the number of standard Young tableaux (SYTs) in a $(2,1)$-hook with the difference of the first two parts fixed, which refine Regev's results in 2009. We also give new combinatorial proofs of Regev's results, and reveal some new recurrence relations related to humps, free Motzkin paths and SYTs.

\end{abstract}

\section{Introduction}
A \emph{Motzkin path} of order $n$ is a lattice path from $(0,0)$ to $(n,0)$, using up steps $(1,1)$ (denoted by $U$), down steps $(1,-1)$ (denoted by $D$) and flat steps $(1,0)$ (denoted by $F$), and never going below the $x$-axis. If such a path is allowed to go below the $x$-axis, it is called a \emph{free Motzkin path}. Let $\mathcal{M}_n$ denote the set of all Motzkin paths of order $n$. The cardinality of $\mathcal{M}_n$ is the $n$th Motzkin number $m_n$ (\cite{Sloane}, A001006).

A \emph{hump} in a Motzkin path is an up step followed by zero or more flat steps followed by a down step. Especially, a hump with no flat steps is called a \emph{peak}. See Figure \ref{fig:0} for an example.
\begin{figure}[ht]
  \centering
  \begin{tikzpicture}[scale=0.6]
    \node[fill,circle,scale=0.5] at (0,0) (A1){};
    \node[fill,circle,scale=0.5] at (1,1) (A2){};
    \node[fill,circle,scale=0.5] at (2,1) (A3){};
    \node[red,fill,circle,scale=0.5] at (3,2) (A4){};
    \node[red,fill,circle,scale=0.5] at (4,2) (A5){};
    \node[red,fill,circle,scale=0.5] at (5,2) (A6){};
    \node[fill,circle,scale=0.5] at (6,1) (A7){};
    \node[blue,fill,circle,scale=0.5] at (7,0) (A8){};
    \node[blue,fill,circle,scale=0.5] at (8,1) (A9){};
    \node[blue,fill,circle,scale=0.5] at (9,0) (A10){};

    \node[above=5pt] at (4,2) (){$P_1$};
    \node[above=5pt] at (8,1) (){$P_2$};
    \draw plot coordinates {(A1) (A2) (A3)};
    \draw[red] plot coordinates {(A3) (A4) (A5) (A6) (A7)};
    \draw plot coordinates {(A7) (A8)};
    \draw[blue] plot coordinates {(A8) (A9) (A10)};
  \end{tikzpicture}
  \caption{An example of hump and peak, where the hump $P_1$ is colored red, and the peak $P_2$ is colored blue.}
\label{fig:0} 
\end{figure}
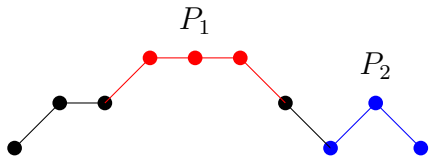

Let $H_{n}$ denote the total number of humps in all Motzkin paths of order $n$. Let $S_n$ denote the number of free Motzkin paths of order $n$. By applying the WZ method \cite{Petkovesk,Zeilberger}, Regev \cite{Regev1} proved that 
\begin{equation}\label{equ:1}
  H_n=\frac{1}{2}(S_n-1)=\frac{1}{2}\left(\sum_{j\geq 1}\binom{n}{j}\binom{n-j}{j}\right).
\end{equation}
A combinatorial proof of (\ref{equ:1}) was given by Ding and Du \cite{Ding}, and generalizations of this study to $(k,a)$-paths were given in \cite{Du,Mansour,Yan}.

Regev pointed out that (\ref{equ:1}) is also related to the number of standard Young tableaux (SYT) in a $(2,1)$-hook, where a $(k,l)$-hook is a partition $\lambda=(\lambda_1,\lambda_2,\dots)$ with $\lambda_{k+1}\leq l$. Let $\SYT(k,l;n)$ denote the number of SYTs of order $n$ in a $(k,l)$-hook. It is obvious that $\SYT(1,1;n)=2^{n-1}$. For the ``hook'' sums of general shapes, only the following result of Regev \cite{Regev1,Regev2} is known.
\begin{equation}\label{equ:2}
  \SYT(2,1;n)-1=H_n=\frac{1}{2}\left(\sum_{j\geq 1}\binom{n}{j}\binom{n-j}{j}\right).
\end{equation}
Du and Yu \cite{Du2} gave a combinatorial proof of (\ref{equ:2}) by proving bijectively that 
$$\SYT(2,1;n)=\frac{1}{2}(S_n+1).$$

The goal of this paper is to reveal some further relations between humps of Motzkin paths and Young tableaux, and obtain some refinements of identities (\ref{equ:1}) and (\ref{equ:2}). The statistic \emph{height} for a hump in a Motzkin path is defined to be the $y$-coordinate reached by the up step of the hump. Note that the number of peaks in Dyck paths with a given height has been studied in several papers, such as \cite{Mansour2} and \cite{Peart}. We denote
$$\mathcal{H}_{n,k}=\{(M,P)\mid M\in \mathcal{M}_n, \text{$P$\ is a hump of $M$ with height $k$}\},$$
$$\mathcal{P}_{n,k}=\{(M,P)\mid M\in \mathcal{M}_n, \text{$P$\ is a peak of $M$ with height $k$}\}.$$
Let $H_{n,k}$ and $P_{n,k}$ denote the cardinality of $\mathcal{H}_{n,k}$ and $\mathcal{P}_{n,k}$ respectively.

To count humps with a given height, we consider the relation between humps and Motzkin prefixes. Here a \emph{Motzkin prefix} is a lattice path that is a prefix of a Motzkin path. We denote $\mathcal{M}_{n,k}$ the set of Motzkin prefixes from $(0,0)$ to $(n,k)$, and denote 
\begin{align*}
  \mathcal{M}_{n,k}^{*U}=&\{M\mid M\in \mathcal{M}_{n,k}\ \text{whose last non-flat step is $U$}\},\\
  \mathcal{M}_{n,k}^{D*}=&\{M\mid M\in \mathcal{M}_{n,k}\ \text{whose first non-up step is $D$}\}.
\end{align*}
See \cite{Krattenthaler} and \cite{Yaqubi} for some further results about Motzkin prefix numbers.  

By constructing a bijection from $\mathcal{H}_{n,k}$ to $\mathcal{M}_{n,2k}^{*U}$, we show that
  \begin{align}
    &H_{n,k}=S_{n,k},\label{equ:3}\\
    &H_{n,k}=\sum\limits_{i=2k-1}^{n-1}M_{i,2k-1},\label{equ:4}\\
    &M_{n,2k}=H_{n,k}+H_{n,k+1}.\label{equ:5}
  \end{align}
where $S_{n,k}$ denotes the number of free Motzkin paths of order $n$ whose smallest $y$-coordinate is $-k$ and last non-flat step is $U$, and $M_{n,k}$ denotes the cardinality of $\mathcal{M}_{n,k}$ (A064189 in OEIS \cite{Sloane}). Note that Equation (\ref{equ:3}) can be viewed as a generalization of Equation (\ref{equ:1}).

Our first main result gives the following formulas for $H_{n,k}$ and $P_{n,k}$.
\begin{theorem}\label{thm:22}
  Let $H_{n,k}$ (resp.\ $P_{n,k}$) be the total number of humps (resp. peaks) with height $k$ in all Motzkin paths of order $n$. Then for $n\geq 2$ and $k\geq 1$, we have
  \begin{align*}
    H_{n,k}&=\sum_{j=0}^{n-2k}\chi(j\equiv n \Mod 2)\frac{4k}{n-j+2k}\binom{n}{j}\binom{n-j-1}{(n-j)/2+k-1},\\
    P_{n,k}&=\sum_{j=0}^{n-2k}\chi(j\equiv n \Mod 2)\frac{4k}{n-j+2k}\binom{n-1}{j}\binom{n-j-1}{(n-j)/2+k-1},
  \end{align*}
  where $\chi$ is the characteristic function. Moreover, the number $j$ in the sum tracks up to the number of flat steps of the Motzkin paths. 
\end{theorem}

We also consider the relation between SYTs contained in a $(2,1)$-hook and Motzkin prefixes. Let $\SYT_k(2,1;n)$ denote the number of SYTs with shape $\lambda$, 
where $\lambda=(\lambda_1,\lambda_2,\dots)$ ranges over all $(2,1)$-hooks with $\mid \lambda\mid=n$ and $\lambda_1-\lambda_2=k$. We first show that
\begin{equation}
  \SYT_{2k-1}(2,1;n+1)+\SYT_{2k-1}(2,1;n)=H_{n+1,k}-H_{n,k}.
\end{equation}
Then in a combinatorial way, we obtain a new proof of (\ref{equ:2}), and obtain the following formula for $\SYT_k(2,1;n)$.
\begin{theorem}\label{thm:33}
  Let $\SYT_k(2,1;n)$ denote the number of SYTs with shape $\lambda$, 
where $\lambda=(\lambda_1,\lambda_2,\dots)$ ranges over all $(2,1)$-hooks with $\mid \lambda\mid=n$ and $\lambda_1-\lambda_2=k$. For integers $n,k$ with $n-2\geq k\geq 0$, we have
    \begin{align*}
      \SYT_k(2,1;n)&=(-1)^{n+k}+\sum_{i=0}^{[\frac{n-k-1}{2}]}\sum_{j=0}^{n-k-1-2i}\chi(j\equiv n+k-1 \Mod 2)\cdot\\
      &\frac{2k+2}{n+k+1-2i-j}\binom{n-2i-2}{j}\binom{n-2i-j-1}{(n+k-j-1)/2-i}.
    \end{align*}
\end{theorem}

\section{The number of humps and peaks with height $k$}\label{sec:2}
A Motzkin prefix $M$ of order $n$ can be represented as a word $M=w_1w_2\cdots w_n$ with $w_i\in\{U,D,F\}$. We use $\overline{M}$ to denote the path obtained from $M$ by reading the steps in reverse order and then swapping the $U$'s and $D$'s.

\begin{lemma}\label{lem:21}
  There is a bijection 
  $$\psi: \bigcup_{1\leq k\leq [\frac{n}{2}]}\mathcal{H}_{n,k}\mapsto \bigcup_{1\leq k\leq [\frac{n}{2}]}\mathcal{M}_{n,2k}^{*U},$$
  such that $(M, P)\in \mathcal{H}_{n,k}$ if and only if $\psi(M, P)\in \mathcal{M}_{n,2k}^{*U}$. Especially, $P$ is a peak of $M$, if and only if $\psi(M, P)$ ends with an up step.
\end{lemma}
\begin{proof}
  Given $(M,P)\in \mathcal{H}_{n,k}$, we assume that $P=UF^rD$, and decompose $M$ as $M=M_1PM_2$. Let  
  $$\psi(M,P)=M_1U\overline{M_2}UF^r.$$
  It is not difficult to see that $\psi(M,P)\in \mathcal{M}_{n,2k}^{*U}$.  Especially, if $M=M_1PM_2$ with $P=UD$ a peak, then $\psi(M,P)=M_1U\overline{M_2}U$ ends with an up step. See Figure \ref{fig:1} for an example of $\psi$.
  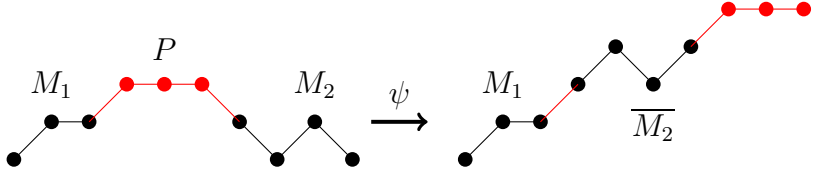
\begin{figure}[ht]
    \centering
    \begin{tikzpicture}[scale=0.5]
      \node[fill,circle,scale=0.5] at (0,0) (A1){};
      \node[fill,circle,scale=0.5] at (1,1) (A2){};
      \node[fill,circle,scale=0.5] at (2,1) (A3){};
      \node[red,fill,circle,scale=0.5] at (3,2) (A4){};
      \node[red,fill,circle,scale=0.5] at (4,2) (A5){};
      \node[red,fill,circle,scale=0.5] at (5,2) (A6){};
      \node[fill,circle,scale=0.5] at (6,1) (A7){};
      \node[fill,circle,scale=0.5] at (7,0) (A8){};
      \node[fill,circle,scale=0.5] at (8,1) (A9){};
      \node[fill,circle,scale=0.5] at (9,0) (A10){};
  
      \node[above=5pt] at (1,1) (){$M_1$};
      \node[above=5pt] at (4,2) (){$P$};
      \node[above=5pt] at (8,1) (){$M_2$};
      \draw plot coordinates {(A1) (A2) (A3)};
      \draw[red] plot coordinates {(A3) (A4) (A5) (A6) (A7)};
      \draw plot coordinates {(A7) (A8) (A9) (A10)};
  
      \path[line width=1.5pt, ->] (9.5,1) edge node [above] {$\psi$} (11,1);
  
      \node[fill,circle,scale=0.5] at (12,0) (B1){};
      \node[fill,circle,scale=0.5] at (13,1) (B2){};
      \node[fill,circle,scale=0.5] at (14,1) (B3){};
      \node[fill,circle,scale=0.5] at (15,2) (B4){};
      \node[fill,circle,scale=0.5] at (16,3) (B5){};
      \node[fill,circle,scale=0.5] at (17,2) (B6){};
      \node[fill,circle,scale=0.5] at (18,3) (B7){};
      \node[red,fill,circle,scale=0.5] at (19,4) (B8){};
      \node[red,fill,circle,scale=0.5] at (20,4) (B9){};
      \node[red,fill,circle,scale=0.5] at (21,4) (B10){};
      \draw plot coordinates {(B1) (B2) (B3)};
      \draw[red] plot coordinates {(B3) (B4)};
      \draw plot coordinates {(B4) (B5) (B6) (B7)};
      \draw[red] plot coordinates {(B7) (B8) (B9) (B10)};
  
      \node[above=5pt] at (13,1) (){$M_1$};
      \node[below=5pt] at (17,2) (){$\overline{M_2}$};
  
    \end{tikzpicture}
    \caption{An example of $\psi$ with $(M,P)\in \mathcal{H}_{9,2}$, where the hump $P$ is colored red.}
  \label{fig:1} 
  \end{figure}

  Conversely, given $M\in \mathcal{M}_{n,2k}^{*U}$, let $M=M_1UM_2UF^r$ be the decomposition of $M$, such that the rightmost endpoint of $M_1$ is the last endpoint of $M$ with $y$-coordinate $k-1$. Then we have $\psi^{-1}(M)=(M_1UF^rD\overline{M_2},P)$, where $P$ is the hump connecting $M_1$ and $\overline{M_2}$.
\end{proof}

Similar to the proof of Lemma \ref{lem:21}, we can give a combinatorial proof of Equation (\ref{equ:3}).
 \begin{proof}[Proof of Equation (\ref{equ:3})]
  Given $(M,P)\in \mathcal{H}_{n,k}$, assume that $P=UF^rD$. We decompose $M$ as $M=M_1PM_2$, and define  
  $$\psi_1(M,P)=M_2DM_1UF^r.$$
  It is not difficult to check that $\psi_1$ is a bijection from $\mathcal{H}_{n,k}$ to the set of free Motzkin paths of order $n$ whose smallest $y$-coordinate is $-k$ and last non-flat step is $U$, which implies Equation (\ref{equ:3}).
 \end{proof} 

Equation (\ref{equ:4}) and (\ref{equ:5}) can be deduced from Lemma \ref{lem:21}.
\begin{proof}[Proof of Equation (\ref{equ:4}) and (\ref{equ:5})]
   Given $P\in\mathcal{M}_{n,2k}^{*U}$, we assume that $P=M_1UH^r$, and define $\psi_2(P)=M_1$. It is not difficult to check that 
   $$\psi_2: \mathcal{M}_{n,2k}^{*U}\longrightarrow \bigcup\limits_{i=2k-1}^{n-1}\mathcal{M}_{i,2k-1}$$
   is a bijection. Thus, Equation (\ref{equ:4}) can be deduced from Lemma \ref{lem:21}. 
 
   On the other hand, given $P\in \mathcal{M}_{n,2k}$, we define $\psi_3(P)$ as follows.
   \begin{itemize}
    \item If $P\in\mathcal{M}_{n,2k}^{*U}$, we define $\psi_3(P)=P$.
    \item If the last non-flat step of $P$ is $D$, we define $\psi_3(P)$ to be the path obtained from $P$ by replacing the last $D$ step with $U$.
   \end{itemize}
  Then $\psi_3$ is a bijection from $\mathcal{M}_{n,2k}$ to the union of $\mathcal{M}_{n,2k}^{*U}$ and $\mathcal{M}_{n,2k+2}^{*U}$, and Equation (\ref{equ:5}) follows.
 \end{proof}

We are now ready to give the proof of Theorem \ref{thm:22}. A \emph{Dyck prefix} is a Motzkin prefix with no flat steps. It is well known that the number of Dyck paths of semilength $n$ ending with $UD^m$ is counted by $\frac{m}{n}\binom{2n-m-1}{n-1}$, which is also equal to the number of Dyck prefixes from $(0,0)$ to $(2n-m,m)$ ending with $U$. (See A33184 in OEIS \cite{Sloane}).

\begin{proof}[Proof of Theorem \ref{thm:22}]
  A path $M\in\mathcal{M}_{n,2k}^{*U}$ with $j$ flat steps can be obtained uniquely by inserting $j$ flat steps into $n-j+1$ positions of a Dyck prefix from $(0,0)$ to $(n-j,2k)$ ending with $U$. Especially, $M$ ends with an up step if and only if the above insertion is not allowed at the end of the path. Combining the above fact and Lemma \ref{lem:21}, we obtain the identities in Theorem \ref{thm:22}.     
\end{proof}
The first few entries of $(H_{n,k})_{n\geq 2,k\geq 1}$ are as follows.
\begin{align*}
  \begin{array}{rrrrrr}
    [2]&1;&&&&\\ \relax
    [3]&3;&&&&\\ \relax
    [4]&8,&1;&&&\\ \relax
    [5]&20,&5;&&&\\ \relax
    [6]&50,&19,&1;&&\\ \relax
    [7]&126,&63,&7;&&\\ \relax
    [8]&322,&196,&34,&1;&\\ \relax
    [9]&834,&588,&138,&9;&\\ \relax
    [10]&2187,&1728,&507,&53,&1.
  \end{array}
\end{align*}

\begin{remark}
  $(H_{n,k})_{n\geq 2,k\geq 1}$ and $(H_{n,1})_{n\geq 1}$ are, respectively, the sequences A379838 and A140662 in OEIS \cite{Sloane}.
\end{remark}

We can also consider the generating function of $(H_{n,k})_{n\geq 2,k\geq 1}$. For fixed $k$, it is well known that the generating function of $(M_{n,k})_{n\geq k}$ is
\begin{equation}\label{equ:2.1}
  \sum_{n\geq k}M_{n,k}x^n=x^kM^{k+1}(x),
\end{equation} 
where $M(x)$ is the generating function of the Motzkin numbers. See A064189 in OEIS \cite{Sloane} for instance. Combining Equation (\ref{equ:5}) and (\ref{equ:2.1}), we obtain the following result.
\begin{corollary}
  The infinite lower triangular array $(H_{n,k})_{n,k\in \mathbb{N}}$ is the Riordan array $(\frac{1}{1-x}, x^2M^2(x))$. Equivalently, for $k\geq 1$, we have
  $$\sum_{n\geq 2k}H_{n,k}x^n=\frac{x^{2k}M^{2k}(x)}{1-x},$$
  where $M(x)$ is the generating function of the Motzkin numbers.
\end{corollary}

\section{The number of SYTs in a $(2,1)$-hook}\label{sec:3}
In this section, our first goal is to give a new combinatorial proof of Equation (\ref{equ:2}). Let $\mathcal{S}_k(2,1;n)$ denote the set of SYTs with shape $\lambda$, 
where $\lambda=(\lambda_1,\lambda_2,\dots)$ ranges over all $(2,1)$-hooks with $\mid \lambda\mid=n$ and $\lambda_1-\lambda_2=k$. Given $T\in \bigcup_{0\leq k\leq n-2}\mathcal{S}_k(2,1;n)$, we define $\phi(T)=w_1w_2\cdots w_n$, where
\begin{equation*}
  w_i=\left\{\begin{aligned}
    U,\ &\text{if $i$ appears in the first row of $T$;}\\
    D,\ &\text{if $i$ appears in the second row of $T$;}\\
    F,\ &\text{otherwise.}
  \end{aligned} \right.
\end{equation*}
It is obvious that
$$\phi: \bigcup_{0\leq k\leq n-2}\mathcal{S}_k(2,1;n) \mapsto \bigcup_{0\leq k\leq n-2}\mathcal{M}_{n,k}^{D*}$$
is a bijection. Therefore, to prove Equation (\ref{equ:2}), it is sufficient to prove the following result.
\begin{lemma}
  There is a bijection
  $$\varphi: \bigcup_{0\leq k\leq n-2}\mathcal{M}_{n,k}^{D*}\mapsto \bigcup_{1\leq k\leq [\frac{n}{2}]}\mathcal{M}_{n,2k}^{*U}.$$
\end{lemma}
\begin{proof}
  Given $0\leq k\leq n-2$ and $M\in \mathcal{M}_{n,k}^{D*}$, we define $\varphi(M)$ as follows.
  \begin{itemize}
    \item If $k$ is odd, we assume that $M=U^{r}DM_1$, and define $\varphi(M)=U^{r-1}FM_1U$.
    \item If $k$ is even, we define $\varphi(M)$ as follows.
    \begin{itemize}
      \item[*] If the last non-flat step of $M$ is $U$, then $\varphi(M)=M$.
      \item[*] If $M=U^{r_1}DM_1DF^{r_2}$, then $\varphi(M)=U^{r_1-1}FM_1UF^{r_2+1}$.
      \item[*] If $M=U^{r_1}DF^{r_2}$, then $\varphi(M)=U^{r_1+1}F^{r_2}$. 
    \end{itemize}
  \end{itemize}

  Conversely, given $1\leq k\leq [\frac{n}{2}]$ and $M\in \mathcal{M}_{n,2k}^{*U}$, we can obtain $\varphi^{-1}(M)$ in the following way.
  \begin{itemize}
    \item If the first non-up step of $M$ is $D$, we have $\varphi^{-1}(M)=M$. 
    \item If the first non-up step of $M$ is $F$ or $M=U^n$, we obtain $\varphi^{-1}(M)$ as follows.
    \begin{itemize}
      \item[*] If $M=U^rFM_1U$, then $\varphi^{-1}(M)=U^{r+1}DM_1$.
      \item[*] If $M=U^{r_1}FM_1UF^{r_2}$ with $r_2\geq 1$, then
       $$\varphi^{-1}(M)=U^{r_1+1}DM_1DF^{r_2-1}.$$
      \item[*] If $M=U^{2k}F^{n-2k}$, then $\varphi^{-1}(M)=U^{2k-1}DF^{n-2k}$.
    \end{itemize}
  \end{itemize} 
\end{proof} 

Combining the three bijections $\psi, \varphi^{-1}$ and $\phi^{-1}$ together, 
we obtain a combinatorial proof of Equation (\ref{equ:2}). See Figure \ref{fig:2} for an example.
\begin{figure}[th]
    \centering
    \begin{tikzpicture}[scale=0.5]
      \node[fill,circle,scale=0.5] at (0,0) (A1){};
      \node[fill,circle,scale=0.5] at (1,1) (A2){};
      \node[fill,circle,scale=0.5] at (2,1) (A3){};
      \node[red,fill,circle,scale=0.5] at (3,2) (A4){};
      \node[red,fill,circle,scale=0.5] at (4,2) (A5){};
      \node[red,fill,circle,scale=0.5] at (5,2) (A6){};
      \node[fill,circle,scale=0.5] at (6,1) (A7){};
      \node[fill,circle,scale=0.5] at (7,0) (A8){};
      \node[fill,circle,scale=0.5] at (8,1) (A9){};
      \node[fill,circle,scale=0.5] at (9,0) (A10){};
  
      \node[above=0pt] at (4,0) () {$(M,P)$};
      \draw[line width=1.5pt, ->] (4,-0.5) --node [right] {$\varPhi$} (4,-2);

      \draw plot coordinates {(A1) (A2) (A3)};
      \draw[red] plot coordinates {(A3) (A4) (A5) (A6) (A7)};
      \draw plot coordinates {(A7) (A8) (A9) (A10)};

      \draw[line width=1.5pt, ->] (9.5,1) --node [above] {$\psi$} (11.5,1);

      \node[fill,circle,scale=0.5] at (12,0) (B1){};
      \node[fill,circle,scale=0.5] at (13,1) (B2){};
      \node[fill,circle,scale=0.5] at (14,1) (B3){};
      \node[fill,circle,scale=0.5] at (15,2) (B4){};
      \node[fill,circle,scale=0.5] at (16,3) (B5){};
      \node[fill,circle,scale=0.5] at (17,2) (B6){};
      \node[fill,circle,scale=0.5] at (18,3) (B7){};
      \node[fill,circle,scale=0.5] at (19,4) (B8){};
      \node[fill,circle,scale=0.5] at (20,4) (B9){};
      \node[fill,circle,scale=0.5] at (21,4) (B10){};
      \draw plot coordinates {(B1) (B2) (B3)};
      \draw plot coordinates {(B3) (B4)};
      \draw plot coordinates {(B4) (B5) (B6) (B7)};
      \draw plot coordinates {(B7) (B8) (B9) (B10)};

      \draw[line width=1.5pt, ->] (16,-0.5) --node [right] {$\varphi^{-1}$} (16,-2);

      \node[fill,circle,scale=0.5] at (12,-5.5) (C1){};
      \node[fill,circle,scale=0.5] at (13,-4.5) (C2){};
      \node[fill,circle,scale=0.5] at (14,-3.5) (C3){};
      \node[fill,circle,scale=0.5] at (15,-4.5) (C4){};
      \node[fill,circle,scale=0.5] at (16,-3.5) (C5){};
      \node[fill,circle,scale=0.5] at (17,-2.5) (C6){};
      \node[fill,circle,scale=0.5] at (18,-3.5) (C7){};
      \node[fill,circle,scale=0.5] at (19,-2.5) (C8){};
      \node[fill,circle,scale=0.5] at (20,-3.5) (C9){};
      \node[fill,circle,scale=0.5] at (21,-3.5) (C10){};

      \draw plot coordinates {(C1) (C2) (C3)(C4) (C5) (C6) (C7)(C8) (C9) (C10)};
  
      \draw[line width=1.5pt, ->]  (11.5,-4)--node [above] {$\phi^{-1}$} (9.5,-4);
      \foreach \x in {2.5,...,6.5}
        \foreach \y in {-3}
        {
          \draw (\x,\y) +(-.5,-.5) rectangle ++(.5,.5);
        }
      \foreach \x in {2.5,...,4.5}
        \foreach \y in {-4}
        {
          \draw (\x,\y) +(-.5,-.5) rectangle ++(.5,.5);
        }

      \draw (2.5,-5) +(-.5,-.5) rectangle ++(.5,.5);

      \draw (2.5,-3) node{1};
      \draw (3.5,-3) node{2};
      \draw (4.5,-3) node{4};
      \draw (5.5,-3) node{5};
      \draw (6.5,-3) node{7};

      \draw (2.5,-4) node{3};
      \draw (3.5,-4) node{6};
      \draw (4.5,-4) node{8};

      \draw (2.5,-5) node{9};
    \end{tikzpicture}
    \caption{An example of $\varPhi$ with $(M,P)\in \mathcal{H}_{9,2}$, where the hump $P$ is colored red.}
  \label{fig:2} 
\end{figure}

Next we calculate $\SYT_k(2,1;n)$. Let 
$$T_{n,k}=\SYT_{k}(2,1;n)+(-1)^{n+k+1},$$
where we set $T_{0,0}=0$. We have the following result for the generating function of $(T_{n,k})_{n\geq k}$ for fixed $k$.
\begin{lemma}\label{lem:33}
  The infinite lower triangular array $(T_{n,k})_{n,k\in \mathbb{N}}$ is the Riordan array $(\frac{xM(x)}{1+x}, xM(x))$. Equivalently, we have
  $$T_k(x)=\sum_{n\geq k} T_{n,k}x^n=\frac{1}{1+x}(xM(x))^{k+1},$$
  where $M(x)$ is the generating function of the Motzkin numbers.
\end{lemma}
\begin{proof}
For $n\geq k$, we first show that
  \begin{equation}\label{rec:31}
    T_{n+1,k}+T_{n,k}=M_{n,k}.
  \end{equation}
It is obvious that $T_{k,k}=0$ and $T_{k+1,k}=1$ for $k\geq 0$, which implies (\ref{rec:31}) for $n=k$. For $n> k$ and $M\in \mathcal{M}_{n,k}$, we define $\varphi_1(M)$ as follows.
\begin{itemize}
  \item If $M\in \mathcal{M}_{n,k}^{D*}$, we define $\varphi_1(M)=M$.
  \item If $M=U^rFM_1$, we define $\varphi_1(M)=U^{r+1}DM_1$.
\end{itemize}
It is not difficult to check that 
$$\varphi_1: \mathcal{M}_{n,k}\mapsto \mathcal{M}_{n+1,k}^{D*}\bigcup \mathcal{M}_{n,k}^{D*}$$ 
is a bijection. Since The left-hand side of (\ref{rec:31}) counts the number of paths in $\mathcal{M}_{n+1,k}^{D*}\bigcup \mathcal{M}_{n,k}^{D*}$, and the right-hand side counts the number of paths in $\mathcal{M}_{n,k}$, we obtain the recurrence relation (\ref{rec:31}).

Combining Equation (\ref{equ:2.1}) and Equation (\ref{rec:31}), we obtain Lemma \ref{lem:33}.
\end{proof}

\begin{remark}
  For $k=0$ and $k=1$, $(T_{n,k})_{n\geq k}$ are, respectively, the sequences A187306 and A284778 in OEIS \cite{Sloane}.
\end{remark}

The proof of Equation (\ref{rec:31}) also implies a combinatorial proof of the following recurrence relation.
\begin{corollary}
For $1\leq k\leq  [\frac{n}{2}]$, we have
  $$\SYT_{2k-1}(2,1;n+1)+\SYT_{2k-1}(2,1;n)=H_{n+1,k}-H_{n,k}.$$
\end{corollary}
\begin{proof}
  By the proof of (\ref{rec:31}), the left-hand side counts the number of paths in $\mathcal{M}_{n,2k-1}$. By Lemma \ref{lem:21}, the right-hand side counts the number of paths in $\mathcal{M}_{n+1,2k}^{*U}$ ending with $U$, which is also equal to the number of paths in $\mathcal{M}_{n,2k-1}$.
\end{proof}

We are now ready to give the proof of Theorem \ref{thm:33}. The weight $w$ of a Motzkin prefix $M\in \mathcal{M}_{n,k}^{*U}$ is defined as $w(M)=(-1)^r$, where $r$ is the number of flat steps after the last up step.
\begin{proof}[Proof of Theorem \ref{thm:33}]
  By Lemma \ref{lem:33}, we have 
$$T_{n,k}=\sum_{M\in \mathcal{M}_{n,k+1}^{*U}}w(M).$$

Given $M\in \mathcal{M}_{n,k+1}^{*U}$ with weight $-1$, we define $\varphi_2(M)$ to be the path obtained from $M$ by moving the last flat step to the beginning. Then $\varphi_2$ is a bijection from the set of paths in $\mathcal{M}_{n,k+1}^{*U}$ with weight $-1$ to those with weight $1$ beginning with a flat step. 

As a consequence of the above bijection, $T_{n,k}$ equals the number of paths in $\mathcal{M}_{n,k+1}^{*U}$ beginning with $U$ and ending with an even number of flat steps. Since those paths can be obtained from the corresponding Dyck prefixes by inserting an even number of flat steps to the end and inserting some flat steps between the first up step and the last up step, we obtain the identity in Theorem \ref{thm:33}.
\end{proof}

The first few entries of $(\SYT_k(2,1;n))_{n,k\in\mathbb{N}}$ are as follows.
\begin{align*}
  \begin{array}{rrrrrrrrrrr}
    [0]&1;&&&&&&&&\\ \relax
    [1]&0,&1;&&&&&&&\\ \relax
    [2]&1,&0,&1;&&&&&&\\ \relax
    [3]&1,&2,&0,&1;&&&&&\\ \relax
    [4]&3,&3,&3,&0,&1;&&&&\\ \relax
    [5]&6,&9,&6,&4,&0,&1;&&&\\ \relax
    [6]&15,&21,&19,&10,&5,&0,&1;&&\\ \relax
    [7]&36,&55,&50,&34,&15,&6,&0,&1;&\\ \relax
    [8]&91,&141,&139,&99,&55,&21,&7,&0,&1.
  \end{array}
\end{align*}

\begin{remark}
  $(\SYT_0(2,1;n))_{n\geq k}$ and $(\SYT_0(2,1;n))_{n\geq 0}$ are, respectively, the sequences A379893 and A005043 in OEIS \cite{Sloane}.
\end{remark}

\section*{Declaration of Competing Interest}
The authors declare that they have no known competing financial interests or personal relationships that could have appeared to influence the work reported in this paper.

\section*{Acknowledgements}
The author was supported by the Natural Science Foundation of Hunan Province (2021JJ40186). The author is grateful to anonymous for providing helpful advice that improved some of the paper's results.

\end{sloppypar}
\end{document}